\documentclass[10pt]{article}
\usepackage{amssymb,amsmath,amsthm,textcomp}
\usepackage{color}
\newtheorem{satz}{Theorem}[section]

\newtheorem{lemma}{Lemma}[section]

\newtheorem{korollar}{Corollary}[section]

\newcommand{\iR}{\mathbb{R}}
\newcommand{\iN}{\mathbb{N}}

\newcommand{\iC}{\mathbb{C}}

\newcommand{\oH}{\hspace*{0.39em}\raisebox{0.6ex}{\textdegree}\hspace{-0.72em}H}
\DeclareMathOperator*{\esup}{ess\,sup}
\DeclareMathOperator*{\einf}{ess\,inf}

\begin{document}
\begin{center}
{\bf\Large Time fractional diffusion equations: solution concepts, regularity and long-time behaviour}
\end{center}
\begin{center}
{\bf Preprint}
\end{center}
\vspace{0.1em}
\begin{center}
Rico Zacher
\end{center}
\begin{abstract}
In this paper we give a survey of results on various analytical aspects of time fractional diffusion equations.
  We describe the approach via abstract Volterra equations and collect results on strong solutions in the $L_p$ sense.
  We further discuss the concept of weak solutions for equations with rough coefficients and give an account of recent
developments towards a De Giorgi-Nash-Moser theory for such equations. The last part summarizes recent results
on the long-time behaviour of solutions, which turns out to be significantly different from that in the heat equation case.
\end{abstract}
\vspace{0.7em}
\begin{center}
\end{center}

\section{Introduction}
The purpose of this paper is to give a survey of results on various analytical aspects of time fractional diffusion equations.
We discuss different solution concepts such as weak solutions and strong $L_p$-solutions and give an account of recent
developments towards a De Giorgi-Nash-Moser theory for such equations. We also describe some very recent results on the
long-time behaviour of solutions, which turns out to be markedly different from that in the classical parabolic case. 

The prototype of the equations we will look at is given by
\begin{equation} \label{TF}
\partial_t^\alpha (u-u_0)-\Delta u=f,\quad t\in (0,T),\,x\in
\Omega.
\end{equation}
Here $T>0$, $\Omega$ is a domain in $\iR^d$ and $u: [0,T]\times \Omega \rightarrow \iR$ is the unknown. 
Further, $\partial_t^\alpha v$ denotes the Riemann-Liouville fractional derivative of order $\alpha\in (0,1)$ w.r.t.\ time.
For (sufficiently smooth) $v:[0,T]\rightarrow \iR$ it is defined by
\[
\partial_t^\alpha v(t)=\partial_t \big(g_{1-\alpha}\ast v\big)(t),
\]
where $\partial_t$ stands for the usual derivative, $g_\beta$ denotes the standard kernel
\[ g_\beta(t)=\frac{t^{\beta-1}}{\Gamma(\beta)},\quad t>0,\;\beta>0,
\]
and $k\ast v$ denotes the convolution on the positive halfline $\iR_+:=[0,\infty)$,
that is $(k\ast v)(t)=\int_0^t k(t-\tau)v(\tau)\,d\tau$, $t\ge 0$. The functions $u_0$ and $f$ are given data;
$u_0$ plays the role of the initial value for $u$, that is 
\begin{equation} \label{ic}
u|_{t=0}=u_0\quad \mbox{in}\; \Omega.
\end{equation}
We point out that for sufficiently smooth $v:[0,T]\rightarrow \iR$,
\begin{equation} \label{Intro1}
\partial_t^\alpha (v-v(0))=g_{1-\alpha}\ast \partial_t v,
\end{equation}
that is, $\partial_t^\alpha (v-v(0))$ coincides with the Caputo fractional derivative of $v$ of order $\alpha$. The formulation on the left-hand side of \eqref{Intro1} has the advantage that it requires less regularity of $v$.

Replacing the Laplacian in \eqref{TF} by a more general elliptic operator of second order (w.r.t. the spatial variables)
leads to the class of problems we will refer to as {\em time fractional diffusion equations}. A considerable part of this paper
will be concerned with the following problem in divergence form
\begin{equation} \label{divTF}
\partial_t^\alpha (u-u_0)-\mbox{div}\,\big(A(t,x)\nabla u\big)=f,\quad t\in (0,T),\,x\in
\Omega,
\end{equation}
where the coefficient matrix $A\in L_\infty((0,T)\times \Omega;\iR^{d\times
d})$ satisfies a uniform parabolicity condition. Here the main problem consists in proving suitable {\em a priori} estimates.
We will explain how these can be obtained and discuss the corresponding (natural) notion of weak solution.

Let us fix some notation. For a Banach space $X$ we denote by ${\cal B}(X)$ the space of all bounded linear operators 
from $X$ into $X$.  For an
interval $J\subset\iR$, $s>0$, $p\in (1,\infty)$ and a UMD space $X$, by $H^s_p(J;X)$
and $B^s_{pp}(J;X)$ we mean the vector-valued Bessel potential
space resp.\ Besov space of $X$-valued functions on
$J$, see e.g.\ \cite{PrSi,Za05}. For $T>0$ and $s\in (0,1]$ we set
${}_0 H^s_p((0,T);X)=\{g_s \ast h:\,h\in L_2(J;X)\}$. Note that for $s\in (1/p,1]$, 
${}_0 H^s_p((0,T);X)=\{v\in H^s_p((0,T);X):\,v(0)=0\}$, cf.\ \cite{Za05}.
\section{Strong solutions and maximal $L_p$-regularity}
Suppose for the moment that $\Omega=\iR^d$. Convolving \eqref{TF} with the kernel $g_\alpha$ and using the
identity $g_{\alpha}\ast g_{1-\alpha}=1$ we obtain
\begin{equation} \label{LP1}
u-g_\alpha\ast \Delta u=u_0+g_\alpha\ast f, \quad t\in (0,T),\,x\in \iR^d.
\end{equation}
In fact, for sufficiently smooth $u$ we have
\[
g_\alpha \ast \partial_t\big(g_{1-\alpha}\ast [u-u_0]\big)
=\partial_t\big(g_\alpha\ast g_{1-\alpha}\ast [u-u_0]\big)=u-u_0.
\]
Equation \eqref{LP1} can be viewed as an abstract Volterra equation. Take as base space, e.g., $X=L_q(\iR^d)$ with
$q\in (1,\infty)$ and define the operator $A$ with domain $D(A)=H^2_q(\iR^d)$ by $Av=-\Delta v$, $v\in D(A)$.
Setting $h=u_0+g_\alpha\ast f$, equation  \eqref{LP1} can be reformulated as
\begin{equation} \label{Volterra}
u(t)+(g_\alpha \ast Au)(t)=h(t),\quad t\in [0,T],
\end{equation}
where now $u$ is regarded as an $X$-valued function of time. 

There is a rich theory of abstract Volterra equations that generalizes semigroup theory and applies to our situation,
the standard reference being the monograph by Pr\"uss \cite{JanI}, see also \cite{CLS, CNa, CN, Grip1}. The operator $A$ is a sectorial operator with
spectral angle $0$ and the kernel $g_\alpha$ is completely monotone and sectorial with angle $\alpha \pi/2$. Since the
sum of the two angles is less than $\pi$, the equation is parabolic and thus admits a resolvent family $\big(S(t)\big)_{t\ge 0}
\subset {\cal B}(X)$, which is the solution operator in case $h=u_0$ (that is, $u(t)=S(t)u_0$ solves the problem) and 
which, in case $\alpha=1$, coincides with the $C_0$-semigroup generated by $-A$. Depending on the regularity of $h$,
the results from \cite{JanI} immediately give existence and uniqueness in the classical and mild sense.    

Here, we want to consider strong $L_p$-solutions, that is, we ask for maximal $L_p$-regularity. Given a Banach space $X$
and a closed linear operator $A$ with domain $D(A)\subset X$, the time fractional evolution equation (with $\alpha\in (0,1]$)
\begin{equation} \label{maxreg}
\partial_t^\alpha u(t)+Au(t)=f(t) ,\quad t\in J:=(0,T),
\end{equation}
is said to have the property of {\it maximal $L_p$-regularity}, if for each $f\in L_p(J;X)$ equation \eqref{maxreg} possesses a unique solution $u$ in the space ${}_0H^\alpha_p(J;X)\cap L_p(J;D_A)$,
i.e.\ both terms on the left-hand side of \eqref{maxreg} belong to $L_p(J;X)$; here $D_A$ denotes the domain of $A$ equipped with the graph norm. 

In the case $\alpha=1$, important contributions on maximal $L_p$-regularity have been made by
Weis \cite{Weis}, who established an operator-valued version of the Mikhlin Fourier multiplier theorem, and by  Denk, Hieber and Pr\"uss \cite{DHP1}. We also refer to the monograph by Pr\"uss and Simonett \cite{PrSi}. The case $\alpha\in (0,1)$ has been intensively studied by
the author \cite{Za05,Za06}; we also refer to \cite{JanI}. Applying the abstract theory from \cite[Theorem 3.4, Theorem 3.6]{Za05} to the time fractional diffusion equation \eqref{TF} in the full space $\iR^d$ and using results on the interpolation of
Sobolev spaces (see e.g.\ \cite{Trie} and \cite{DHP2}) we obtain the following result.
\begin{satz} \label{MR1}
Let $p,q\in (1,\infty)$ and $\alpha\in (\frac{1}{p},1)$. Then the problem \eqref{TF}, \eqref{ic} with $\Omega=\iR^d$ admits a unique
solution
\[
u\in Z:=H^\alpha_p(J;L_q(\iR^d))\cap L_p(J;H^2_q(\iR^d)),
\]
if and only if $f\in L_p(J;L_q(\iR^d))$ and $u_0\in B_{qp}^{2-\frac{2}{p\alpha}}(\iR^d)$. Furthermore, we have the continuous embedding
\[
Z \hookrightarrow C([0,T];B_{qp}^{2-\frac{2}{p\alpha}}(\iR^d)).
\]
\end{satz}
This result extends to second order elliptic operators in non-divergence form under suitable regularity assumptions
on the coefficients like continuity of the top order coeffcients; the case $q=p$ can be found in \cite{Za06}. Note that
the condition $\alpha>1/p$ ensures that functions $u\in Z$ have a time trace with the Besov space $B_{qp}^{2-\frac{2}{p\alpha}}(\iR^d)$ being the natural trace space. We remark that $L_p(L_q)$-estimates for time fractional diffusion equations 
in $\iR^d$ have also been proved recently in \cite{KKL} by PDE methods. 

In the case $q=p$, there also exist corresponding results for problems on domains with nonhomogenous boundary conditions,
see \cite{Za06}. As an example, we formulate such a result for a Dirichlet boundary condition, i.e.
we consider the problem
\begin{equation} \label{maxreg2}
\left\{
\begin{array}{r@{\;=\;}l@{\;}l}
\partial_t^\alpha (u-u_0)-\Delta u  & f,\quad & t\in (0,T),\,x\in\Omega\\
u|_{\partial \Omega} & g,\quad & t\in (0,T),\,x\in \partial\Omega\\
u|_{t=0} & u_0,\; & x\in \Omega.
\end{array}
\right.
\end{equation}
\begin{satz} \label{MR2} \label{MRRand}
Let $\Omega\subset \iR^d$ ($d\ge 2$) be a domain with compact $C^2$-boundary $\partial\Omega$.
Let $J=(0,T)$, $p\in (1,\infty)$ and assume that $\alpha\in (\frac{1}{p},1)\setminus \{\frac{2}{2p-1}\}$.
Then \eqref{maxreg2} possesses a unique solution $u$ in the space 
$H^\alpha_p(J;L_p(\Omega))\cap L_p(J;H^2_p(\Omega))$ if and only if the functions $f$, $g$, $u_0$ are subject to the following conditions.
\begin{align*}
 & (i)\quad f\in L_p(J;L_p(\Omega)),\quad  (ii) \quad g\in B_{pp}^{\alpha(1-\frac{1}{2p})}(J;L_p(\partial
\Omega))\cap L_p(J;B_{pp}^{2-\frac{1}{p}}(\partial\Omega)),\\
& (iii) \quad u_0\in B_{pp}^{2-\frac{2}{p\alpha}}(\Omega), \quad  (iv) \quad
g|_{t=0}=u_0|_{\partial \Omega}\;\;\mbox{if}\;\;\alpha>\frac{2}{2p-1}.
\end{align*} 
\end{satz} 
To prove this theorem, one can use the localization method and perturbation arguments to reduce the problem to related problems on the full space $\iR^d$ and the half space 
$\iR^d_+=\{(x',y)\in \iR^d: x'\in \iR^{d-1}, y>0\}$. These problems in turn can then be treated by means of operator theoretic methods, see \cite{Za06}.
\section{Weak solutions in the Hilbert space setting} \label{weaksetting}
We turn now to weak solutions. Let $T>0$ and $\Omega$ be a bounded domain in $\iR^d$. We consider the problem
\begin{equation} \label{weakProb}
\left\{
\begin{array}{r@{\;=\;}l@{\;}l}
\partial_t^\alpha (u-u_0)-\mbox{div}\,(A\,\nabla u)+cu &  f,\quad &
 t\in (0,T),\,x\in \Omega\\
u|_{\partial\Omega} &  0, \quad &  t\in (0,T),\,x\in \partial\Omega\\
u|_{t=0} &  u_0,\quad &  x\in \Omega.
\end{array}
\right.
\end{equation}
The coefficients and data are supposed to satisfy the following assumptions.
\begin{itemize}
\item [{\bf (Hd)}] $\; u_0\in L_2(\Omega)$, $f\in L_2((0,T);L_2(\Omega))$, $c\in L_\infty((0,T)\times
\Omega)$.
\item [{\bf (HA)}] $\;A\in L_\infty((0,T)\times \Omega;\iR^{d\times
d})$, and there exists a $\nu>0$ such that
\[
(A(t,x)\xi|\xi)\ge \nu|\xi|^2,\quad \mbox{for a.a.}\;t\in
(0,T),\,x\in \Omega,\,\mbox{and all}\,\xi\in \iR^d.
\]
\end{itemize}
Here $(\cdot|\cdot)$ denotes the standard scalar product in $\iR^d$. 

In what follows we denote by $y_+$ and $y_-:=[-y]_+$ the positive and negative part, respectively, of $y\in \iR$. We say that $u$ is a {\em weak solution (subsolution, supersolution)} of \eqref{weakProb} if
\begin{itemize}
\item[(a)] $u\in
W:=\{w\in  L_2((0,T);H^1_2(\Omega)):\;g_{1-\alpha}\ast w\in C([0,T];L_2(\Omega))\;\mbox{and}\;
(g_{1-\alpha}\ast w)|_{t=0}=0\}$,
\item[(b)]
$u\,\big(u_+,\; u_-\big) \in L_2((0,T);\oH^1_2(\Omega))$, where $\oH^1_2(\Omega):=\overline{C_0^\infty(\Omega)}\,{}^{H^1_2(\Omega)}$,
\item[(c)] for any nonnegative test function
\[
\eta\in H^1_2((0,T);L_2(\Omega))\cap
L_2((0,T);\oH^1_2(\Omega))
\]
with $\eta|_{t=T}=0$ there holds
\[
\int_{0}^{T} \int_\Omega \Big(-\eta_t \big(g_{1-\alpha}\ast [u-u_0]\big)+
(A\nabla u|\nabla \eta)+cu\eta \Big)\,dx\,dt=\,(\le,\,\ge)\,\int_0^T \int_\Omega f\eta \,dx\,dt.
\]
\end{itemize}
The following theorem is due to the author, see \cite[Section 4]{ZWH}. Here the symbol $L_{p,\infty}$
refers to the weak $L_p$ space and $H^{-1}_2(\Omega)$
denotes the dual space of $\oH^1_2(\Omega)$.
\begin{satz} \label{weaksol}
Let $T>0$ and $\Omega$ be a bounded domain in $\iR^d$. Let $\alpha\in (0,1)$ and assume
that (Hd) and (HA) hold. Then the problem \eqref{weakProb} has a unique weak solution $u\in W$ and 
\[
|g_{1-\alpha} \ast u|_{C([0,T];L_2(\Omega))}+|u|_{ L_2((0,T);H^1_2(\Omega))}\le C\big (|u_0|_{L_2(\Omega)}+
|f|_{L_2((0,T);L_2(\Omega))}\big),
\]
where the constant $C$ is independent of $u$, $u_0$, and $f$. Moreover, we have
\begin{equation} \label{extrareg}
u\in L_{\frac{2}{1-\alpha},\infty}((0,T);L_2(\Omega))\quad\mbox{and}\quad u-u_0\in {}_0 H^\alpha_2
((0,T);H_2^{-1}(\Omega)).
\end{equation}
\end{satz}
Note that $u\in W$ does not entail
$u\in C([0,T];L_2(\Omega))$ in general, so it is not so clear how to interpret the
initial condition. However, once one knows that the solution $u$ is sufficiently smooth (e.g. if $\alpha>\frac{1}{2}$), then $u|_{t=0}=u_0$ is satisfied in an appropriate sense (see \cite{ZWH}). We also point out
that the statement of Theorem \ref{weaksol} remains true, if we only assume that $f\in L_2((0,T); H_2^{-1}(\Omega))$; the integral $\int_\Omega f\eta \,dx$ in the weak formulation above then has to be replaced
by the duality pairing $\langle
f,\eta \rangle$ between $H_2^{-1}(\Omega)$ and  $\oH^1_2(\Omega)$.

The first statement in \eqref{extrareg} follows from considerations for more general problems (cf.\ 
\cite{ZWH}) and can be
slightly improved in the time fractional case. In fact, the solution $u$ even enjoys the property
\begin{equation} \label{extrareg2}
u\in L_{\frac{2}{1-\alpha}}((0,T);L_2(\Omega)),
\end{equation}
which is also in accordance with the estimates in \cite{ACV2} for weak solutions of bifractional porous
medium equations, see also \cite{ACV}. To see \eqref{extrareg2}, we use Theorem \ref{weaksol}, cross interpolation 
(see e.g.\ the mixed derivative theorem in \cite{Sob}) and Sobolev embedding, thereby obtaining that
\begin{align*}
u\in &  H^\alpha_2((0,T);H_2^{-1}(\Omega))\cap L_2((0,T);H^1_2(\Omega)) \\
& \hookrightarrow
H^{\frac{\alpha}{2}}_2((0,T);L_2(\Omega)) \hookrightarrow L_{\frac{2}{1-\alpha}}((0,T);L_2(\Omega)).
\end{align*}

Theorem \ref{weaksol} follows from a rather general result on weak solutions for abstract evolutionary
integro-differential equations in Hilbert spaces (see \cite[Theorem 3.1]{ZWH}), which is the non-local in time analogue of the classical result on weak solutions for abstract parabolic equations given via
a bounded and coercive bilinear form, cf. e.g. Theorem 4.1 and Remark 4.3 in Chapter 4 in Lions and Magenes \cite{LM} or Zeidler
\cite[Section 23]{ZeidlerII}. The theory from \cite{ZWH} covers a wide range of non-local in time subdiffusion problems, including also problems with sums of fractional derivatives and ultra-slow diffusion
equations (cf.\ \cite{Koch08}) and with other boundary conditions like a Neumann boundary condition.  

The proof of Theorem 3.1 in \cite{ZWH} is based on the Galerkin method and
suitable {\em a priori} estimates, which can be derived by means of a basic identity for integro-differential
operators $B$ of the form $Bv=\partial_t(k\ast v)$. 
Before explaining several versions of this so-called {\em fundamental identity} we collect some further
basic results on \eqref{weakProb}. 

The first is the weak maximum principle for (\ref{weakProb}) with $f=0$. It is contained
in \cite[Theorem 3.2]{Za}, which also covers the case of non-homogenous boundary data and more
general subdiffusion equations. Its proof relies on the fundamental identity described in the next section.
\begin{satz} \label{MaxPr1} 
Let $T>0$ and $\Omega$ be a bounded domain in $\iR^d$. Let $\alpha\in (0,1)$ and assume
that (Hd) and (HA) are satisfied. Assume further that $f=0$ and $c\ge 0$.
Then for any weak subsolution (supersolution) $u$ of (\ref{weakProb}) there holds for a.a.\
$(t,x)\in (0,T)\times \Omega$
\[
u(t,x)\le \max\big\{0,\esup_\Omega u_0\big\} \quad\quad \Big(\; u(t,x)\ge \min\big\{0,\einf_\Omega u_0\big\}\;\Big),
\]
provided this maximum (minimum) is finite.
\end{satz}
Results on the maximum principle in a stronger setting have also been found in \cite{Luch11,Luch09} by 
different methods.  

The next result provides the comparison principle for (\ref{weakProb}). It is a special case of 
\cite[Theorem 3.3]{VZ2}. 
\begin{satz} \label{comp}
Let $T>0$ and $\Omega$ be a bounded domain in $\iR^d$. Let $\alpha\in (0,1)$ and assume
that (Hd) and (HA) are satisfied. Suppose that $u\in W$ is a weak subsolution of (\ref{weakProb}) and that $v\in W$ is a weak supersolution of (\ref{weakProb}). Then $u\le v$ a.e.\ in $(0,T)\times \Omega$.
\end{satz}
The comparison principle for (\ref{weakProb}) has also been proved in \cite{LuchYam} under much stronger
assumptions; e.g.\ in \cite{LuchYam}, the coefficient matrix $A$ may only depend on $x$ and has to be symmetric
as well as $C^1(\bar{\Omega})$-smooth.

We remark that weak solutions for time fractional diffusion equations with nonhomogenous Dirichlet boundary
condition have been studied recently in \cite{Yam18}. 
\section{The fundamental identity}
An important tool for deriving {\em a priori} estimates for time fractional diffusion equations (in particular in the weak setting)
is the so-called
{\em fundamental identity} for integro-differential
operators of the form $\partial_t(k\ast \cdot)$, see e.g.\
\cite{Za1}. It can be viewed as the analogue to the chain rule
$(H(u))'=H'(u)u'$. 
The time derivative of $k$ (also in the generalized sense) will be denoted by $\dot{k}$. 
\begin{lemma} \label{FILemma}
Let $T>0$, $J=(0,T)$ and $U$ be an open subset of $\iR$. Let further $k\in
H^1_1(J)$, $H\in C^1(U)$, and $u\in L_1(J)$ with $u(t)\in U$
for a.a. $t\in J$. Suppose that the functions $H(u)$, $H'(u)u$,
and $H'(u)(\dot{k}\ast u)$ belong to $L_1(J)$ (which is the case
if, e.g., $u\in L_\infty(J)$). Then we have for a.a. $t\in J$,
\begin{align} \label{fundidentity}
H'(u(t))&\partial_t(k \ast u)(t) =\;\partial_t\big(k\ast
H(u)\big)(t)+
\Big(-H(u(t))+H'(u(t))u(t)\Big)k(t) \nonumber\\
 & +\int_0^t
\Big(H(u(t-s))-H(u(t))-H'(u(t))[u(t-s)-u(t)]\Big)[-\dot{k}(s)]\,ds.
\end{align}
\end{lemma}
The proof is a straightforward computation. We remark that (\ref{fundidentity}) remains valid for singular kernels $k$, like e.g.\ $k=g_{1-\alpha}$ with
$\alpha\in(0,1)$, provided that $u$ is sufficiently smooth. Recalling that $\partial_t^\alpha u=\partial_t(g_{1-\alpha}\ast u)$,
in the latter case (\ref{fundidentity}) thus applies to the Riemann-Liouville fractional derivative. 

In the weak setting, a key idea is to reformulate the problem in such a way that the fractional derivative is replaced
by its {\em Yosida approximations}, which take the form $B_n u=\partial_t(g_{1-\alpha,n}\ast u)$, $n\in \iN$, where 
$g_{1-\alpha,n}=n s_n$ (see Section \ref{bdddomain} below for the definition of $s_n$) is nonnegative, nonincreasing and
belongs to $H^1_1(J)$ for each $T>0$. We refer to \cite{VZ1} for the computation of the Yosida approximation and to 
\cite{Za} for the derivation of (in time) regularized weak formulations.

A direct consequence of the fundamental identity is the following {\em convexity inequality} (cf. \cite{KSVZ}), which is in particular very useful
when dealing with fractional derivatives in the Caputo sense.
\begin{korollar} \label{convexFI} Let $T, J, U, k, H$, and $u$ be as in Lemma \ref{FILemma}. Let $u_0\in \iR$, and assume in addition
that $k$ is nonnegative and nonincreasing and that $H$ is convex. Then
\begin{align} \label{convexfundidentity}
H'(u(t))&\partial_t\big(k \ast [u-u_0]\big)(t) \ge \;\partial_t\big(k\ast
[H(u)-H(u_0)]\big)(t),\quad \mbox{a.a.}\;t\in J.
\end{align}
\end{korollar}
\begin{proof} By the fundamental identity, convexity of $H$, and the properties of $k$, we have for a.a.\ $t\in J$
\begin{align*}
H'(u(t))\partial_t& \big(k \ast [u-u_0]\big)(t)=H'(u(t))\partial_t\big(k \ast u\big)(t)-H'(u(t))u_0 k(t)\\
& \ge \partial_t\big(k\ast
H(u)\big)(t)+
\Big(-H(u(t))+H'(u(t))[u(t)-u_0]\Big)k(t)\\
& \ge \partial_t\big(k\ast
H(u)\big)(t)-H(u_0)k(t)\\
& =\,\partial_t\big(k\ast
[H(u)-H(u_0)]\big)(t),
\end{align*}
which shows the asserted inequality.
\end{proof}
Another important consequence of Lemma \ref{FILemma} is the so-called {\em $L_p$-norm inequality} for operators of the form $\partial_t(k\ast\cdot)$, which has been established
in Vergara and Zacher \cite{VZ}. In the special case $p=2$ it states the following.
\begin{satz} \label{Lpnorm}
Let $T>0$ and $\Omega\subset \iR^d$ be an open set. Let $k\in H^1_{1,loc}(\iR_+)$ be nonnegative and nonincreasing. Then for any $v\in L_2((0,T)\times \Omega)$ and any $v_0\in L_2(\Omega)$ we have for a.a.\ $t\in (0,T)$
\begin{equation} \label{LPU}
\int_{\Omega}v\partial_t\big(k \ast [v-v_0]\big)\,dx\ge |v(t)|_{L_2(\Omega)}
\partial_t\big(k\ast \big[|v|_{L_2(\Omega)}-|v_0|_{L_2(\Omega)}\big]\big)(t).
\end{equation}
\end{satz}
\begin{proof} The following argument is simpler
than that in the more general case considered in \cite{VZ}.

By the fundamental identity, applied twice (!), Fubini's theorem, and the triangle inequality for the
$L_2(\Omega)$-norm we have for a.a.\ $t\in (0,T)$
\begin{align*}
\int_{\Omega}v \partial_t(k\ast v)\,dx = &\,\int_{\Omega} \Big(\frac{1}{2}\,\partial_t
(k\ast v^2)+\frac{1}{2}\,k(t)v^2\Big)\,dx\\
&\,+\int_0^t \int_{\Omega} |v(t,x)-v(t-s,x)|^2\,dx\, [-\dot{k}(s)]\,ds\\
\ge &\,\,\frac{1}{2}\,\partial_t\big(k\ast |v|^2_{L_2(\Omega)})+\,\frac{1}{2}\,k(t)
|v(t)|^2_{L_2(\Omega)}\\
&\,+\,\frac{1}{2}\,\int_0^t \Big(|v(t)|_{L_2(\Omega)}-|v(t-s)|_{L_2(\Omega)}
\Big)^2   [-\dot{k}(s)]\,ds\\
= &\,\, |v(t)|_{L_2(\Omega)} \partial_t\big(k\ast |v|_{L_2(\Omega)}\big)(t).
\end{align*}
From this and H\"older's inequality, we infer that for a.a.\ $t\in (0,T)$
\begin{align*}
\int_{\Omega}v\partial_t\big(k \ast [v-v_0]\big)\,dx & = \int_{\Omega}v\partial_t\big(k \ast v\big)\,dx-k(t)\int_\Omega vv_0\,dx\\
& \ge |v(t)|_{2} \partial_t\big(k\ast |v|_{2}\big)(t)-k(t)|v(t)|_2 |v_0|_2\\
& = |v(t)|_{2} \partial_t\big(k\ast [|v|_{2}-|v_0|_2]\big)(t).
\end{align*}
This proves the theorem.
\end{proof}
As in the case of the fundamental identity, Corollary \ref{convexFI} and Theorem \ref{Lpnorm} extend to singular kernels
$k$ including $k=g_{1-\alpha}$ for sufficiently smooth functions.

The following identity is basic to energy estimates in the Hilbert
space setting. For ${\cal H}=\iR$ it coincides with
(\ref{fundidentity}) with $H(y)=\frac{1}{2}y^2$, $y\in \iR$. See \cite{VZ1}.
\begin{lemma} \label{fundlemma1}
Let ${\cal H}$ be a real Hilbert space with scalar product
$(\cdot,\cdot)_{\cal H}$ and $T>0$. Then for any $k\in
H^1_1((0,T))$ and any $v\in L_2((0,T);{\cal H})$ there holds
\begin{align}
\Big(\partial_t(k\ast v)(t), & v(t)\Big)_{{\cal H}} =
\frac{1}{2}\,\partial_t(k\ast
|v(\cdot)|_{\cal H}^2)(t)+\frac{1}{2}\,k(t)|v(t)|_{\cal H}^2 \nonumber\\
&\;+\,\frac{1}{2}\,\int_0^t [-\dot{k}(s)]\,|v(t)-v(t-s)|_{\cal
H}^2\,ds,\quad \mbox{a.a.}\;t\in(0,T).\label{ident1}
\end{align}
\end{lemma}
\section{De Giorgi-Nash-Moser estimates}
We consider again the time fractional diffusion equation from \eqref{weakProb} and set $c=0$
for simplicity, that is we look at 
\begin{equation} \label{divTF2}
\partial_t^\alpha (u-u_0)-\mbox{div}\,\big(A(t,x)\nabla u\big)=f,\quad t\in (0,T),\,x\in
\Omega.
\end{equation}
As before, the coefficient matrix $A$ is merely assumed to satisfy condition (HA). That is, we do not
assume any regularity on the coefficients; in this situation one also speaks of {\em rough coefficients}.

In the elliptic and in the classical parabolic
case (i.e. in the case $\alpha=1$) there is a powerful theory of {\em a priori} estimates, often referred to
as {\em De Giorgi-Nash-Moser theory}, which provides local and global estimates for weak solutions of the
respective equations such as local and global boundedness, Harnack and weak Harnack inequalities as
well as H\"older continuity of weak solutions, see \cite{GilTrud, HanLin} for the elliptic and
\cite{LSU,Lm} for the parabolic case. H\"older estimates
are of utmost importance for the study of quasilinear problems.
In fact, in the elliptic case their discovery opened up the theory
of quasilinear equations in higher dimensions; in the parabolic
case they allow to prove global in time existence. 

Since the time fractional case with $\alpha\in (0,1)$ can be
viewed in some sense as an intermediate case between the elliptic
($\alpha=0$) and the classical parabolic case ($\alpha=1$), one
might think that corresponding results can also be obtained
in the time fractional situation. However, there is a significant
difference to the cases $\alpha=0$ and $\alpha=1$: the time
fractional equations are {\em non-local} in time, due to the non-local
nature of the operator $\partial_t^\alpha$. This feature
complicates the matter considerably, as the theory described above
heavily relies on {\em local} estimates. Another difficulty
consists in the lack of a simple calculus for integro-differential
operators like $\partial_t^\alpha$. Instead of the simple chain rule for the usual derivative,
one has to employ the fundamental identity from the previous section in order to use the test-function
method, the latter being the basic tool for deriving {\em a priori}
bounds for weak solutions of equations in divergence form.

In the following we will describe the main results of the De Giorgi-Nash-Moser theory
for \eqref{divTF2}, which has been developed by the author, see \cite{Za,Za1,ZaG,ZwH}, see also
\cite{ACV,ACV2} for the fully non-local case.

Throughout this section we will assume that $\alpha\in (0,1)$, $u_0\in L_2(\Omega)$ and that the function $f$ satisfies the following condition.
\begin{itemize}
\item [{\bf (Hf)}] $\quad f\in
L_r((0,T);L_q(\Omega))$, where $r,q\ge 1$ fulfil
\end{itemize}
\[
\frac{1}{\alpha r}\,+\,\frac{d}{2q}\,=1-\kappa,
\]
and
\[
\begin{array}{l@{\quad \mbox{for}\quad}l}
r\in \Big[\,\frac{1}{\alpha(1-\kappa)}\,,\infty\Big],\;
q\in \Big[\,\frac{d}{2(1-\kappa)}\,,\infty\Big],\;\kappa\in (0,1) & d\ge 2,\\
r\in
\Big[\,\frac{1}{\alpha(1-\kappa)}\,,\frac{2}{\alpha(1-2\kappa)}\,\Big],\;
q\in [1,\infty],\;\kappa\in \Big(0,\frac{1}{2}\Big) & d=1.
\end{array}
\]

We say that a function $u$ is a {\em weak solution (subsolution,
supersolution)} of \eqref{divTF2} in $(0,T)\times \Omega$, if $u$ belongs to
the space
\begin{align*}
{W}_\alpha:=\{&\,v\in
L_{\frac{2}{1-\alpha},\,\infty}((0,T);L_2(\Omega))\cap
L_2((0,T);H^1_2(\Omega))\;
\mbox{such that}\;\\
&\;\;g_{1-\alpha}\ast v\in C([0,T];L_2(\Omega)),
\;\mbox{and}\;(g_{1-\alpha}\ast v)|_{t=0}=0\},
\end{align*}
and for any nonnegative test function
\[
\eta\in H^1_2((0,T);L_2(\Omega))\cap
L_2((0,T);\oH^1_2(\Omega))
\]
with $\eta|_{t=T}=0$ there holds
\[
\int_{0}^{T} \int_\Omega \Big(-\eta_t \big(g_{1-\alpha}\ast [u-u_0]\big)+
(A\nabla u|\nabla \eta)\Big)\,dx\,dt=\,(\le,\,\ge)\,\int_0^T \int_\Omega f\eta \,dx\,dt.
\]
We point out that here \eqref{divTF2} is considered without any boundary conditions. In this sense,
weak solutions of \eqref{divTF2} as defined just before are {\em local} ones (w.r.t.\ $x$). Note that this
weak formulation is consistent with the one from Section \ref{weaksetting}, in view of Theorem 
\ref{weaksol}.

Before stating the first result on global boundedness we need some preliminaries.

The boundary $\partial\Omega$ of a bounded domain $\Omega\subset \iR^d$ is said to satisfy the property of {\em
positive geometric density}, if there exist $\beta\in (0,1)$ and
$\rho_0>0$ such that for any $x_0\in \partial\Omega$, any open ball $B(x_0,\rho)$
with $\rho\in (0,\rho_0]$ we have that $\lambda_d(\Omega\cap
B(x_0,\rho))\le \beta\lambda_d(B(x_0,\rho))$, where $\lambda_d$ denotes the Lebesgue measure
in $\iR^d$, cf.\ e.g.\
\cite[Section I.1]{DB}.

\noindent In what follows we say that a function $u\in W_\alpha$
satisfies $u\le K$ a.e.\ on $(0,T)\times \partial\Omega$ for some number $K\in \iR$ if
$(u-K)_+\in L_2((0,T);\oH^1_2(\Omega))$, likewise for lower bounds
on $(0,T)\times \partial\Omega$.

The subsequent result provides sup-bounds for weak subsolutions. The proof given in \cite{Za}
uses De Giorgi's iteration technique.
\begin{satz} \label{GlobalBddness}
Let $T>0$ and $\Omega\subset\iR^d$ be a bounded domain satisfying the property of
positive geometric density. Let further $\alpha\in (0,1)$, $u_0\in L_2(\Omega)$ and assume that
the conditions (HA) and (Hf) are satisfied. Suppose $K\ge
0$ is such that $u_0\le K$ a.e. in $\Omega$. Then there exists a
constant $C=C(\alpha,q,r,T,d,\nu,\Omega,f)$ such
that for any weak subsolution $u\in W_\alpha$ of \eqref{divTF2}
in $(0,T)\times \Omega$ satisfying $u\le K$ a.e. on $(0,T)\times \partial \Omega$ there holds
$u\le C(1+K)$ a.e. in $(0,T)\times \Omega$.
\end{satz}
There is a corresponding result for weak supersolutions $u$ of \eqref{divTF2}
in the situation where $u_0\ge K$ a.e. in $\Omega$
and $u\ge K$ a.e. on $(0,T)\times \partial\Omega$, for some $K\le 0$. This follows
immediately from Theorem \ref{GlobalBddness} by replacing $u$ with
$-u$, and $u_0$ with $-u_0$. As shown in \cite{Za}, Theorem \ref{GlobalBddness} extends to a wide
class of subdiffusion equations. 

As an immediate consequence of Theorem
\ref{GlobalBddness} and the remark following it we obtain the
global boundedness of weak solutions of \eqref{divTF2} that are
bounded on the parabolic boundary of $(0,T)\times \Omega$.
\begin{korollar} \label{Linftybound}
Let $T>0$ and $\Omega\subset\iR^d$ be a bounded domain satisfying the property of
positive geometric density. Let further $\alpha\in (0,1)$, $u_0\in L_\infty(\Omega)$ and assume that
the conditions (HA) and (Hf) are satisfied. Suppose $K\ge 0$
is such that $|u_0|\le K$ a.e. in $\Omega$. Then there exists a
constant $C=C(\alpha,q,r,T,d,\nu,\Omega,f)$ such
that for any weak solution $u\in W_\alpha$ of \eqref{divTF2}
in $(0,T)\times \Omega$ which satisfies $u\le K$ a.e. on $(0,T)\times \partial \Omega$ we have
$|u|\le C(1+K)$ a.e. in $(0,T)\times \Omega$.
\end{korollar}

We turn now to H\"older regularity of bounded weak solutions. For $\beta_1,\beta_2\in (0,1)$ and $Q\subset 
(0,T)\times \Omega$ we set
\[
[u]_{C^{\beta_1,\beta_2}(Q)}:=\sup_{(t,x),(s,y)\in
{Q},\,(t,x)\neq(s,y)}\Big\{
\frac{|u(t,x)-u(s,y)|}{|t-s|^{\beta_1}+|x-y|^{\beta_2}} \Big\}.
\]
The main regularity theorem reads as follows, see Zacher \cite{Za1}.
\begin{satz} \label{interiorHoelder}
Let $\alpha\in (0,1)$, $T>0$ and $\Omega$ be a bounded domain in
$\iR^d$. Let $u_0\in L_\infty(\Omega)$ and suppose that the assumptions (HA) and (Hf) are satisfied. Let
$u\in W_\alpha$ be a bounded weak solution of \eqref{divTF2}
in $(0,T)\times \Omega$. Then there holds for any
$Q\subset (0,T)\times \Omega$ separated from the parabolic boundary $(\{0\}\times\Omega)\cup((0,T)\times \partial\Omega)$ by a positive
distance $D$,
\[
[u]_{C^{\frac{\alpha\epsilon}{2},\epsilon}(\bar{Q})}\le
C\Big(|u|_{L_\infty((0,T)\times \Omega)}+|u_0|_{L_\infty(\Omega)}+|f|_{L_r((0,T);L_q(\Omega))}\Big)
\]
with positive constants
$\epsilon=\epsilon(|A|_\infty,\nu,\alpha,r,q,d,\mbox{diam}\,\Omega,\inf_{(\tau,z)\in
Q}\tau)$ and $C=C(|A|_\infty,\nu,\alpha,r,q,d$,
$\mbox{diam}\,\Omega,\lambda_{d+1}(Q),D)$.
\end{satz}
Theorem \ref{interiorHoelder} gives an interior H\"older estimate
for bounded weak solutions of \eqref{divTF2} in terms of the data
and the $L_\infty$-bound of the solution. It can be viewed as the
time fractional analogue of the classical parabolic version
($\alpha=1$) of the celebrated De Giorgi-Nash theorem on the
H\"older continuity of weak solutions to elliptic equations in
divergence form (De Giorgi \cite{DeGiorgi}, Nash \cite{Nash}), see
also \cite{GilTrud} for the elliptic,
and \cite{LSU} as well as the seminal contribution by
Moser \cite{Moser64} for the parabolic case.

The proof of Theorem \ref{interiorHoelder} is quite involved. It uses De
Giorgi's technique and the method of {\em non-local growth
lemmas}, which has been developed in \cite{Sil} for
integro-differential operators like the fractional Laplacian. The
fundamental identity is frequently used to
derive various a priori estimates for $u$ and certain logarithmic
expressions involving $u$.

The following result gives conditions on the data which are
sufficient for H\"older continuity up to the parabolic boundary of $(0,T)\times
\Omega
$. It has been taken from \cite{ZaG}.
\begin{satz} \label{Regparaboundary}
Let $\alpha\in (0,1)$, $T>0$, $d\ge 2$, and $\Omega\subset \iR^d$
be a bounded domain with $C^2$-smooth boundary $\partial\Omega$. Let the
assumptions (HA) and (Hf) be satisfied. Suppose further that
\[
u_0\in B_{pp}^{2-\frac{2}{p\alpha}}(\Omega),\quad g\in Y:=
B_{pp}^{\alpha(1-\frac{1}{2p})}((0,T);L_p(\partial\Omega))\cap
L_p((0,T);B_{pp}^{2-\frac{1}{p}}(\partial\Omega))
\]
for some $p>\frac{1}{\alpha}+\frac{d}{2}$, and that the compatibility
condition
\[
u_0|_{\partial\Omega}=g|_{t=0}\quad \mbox{on}\;\,\,\partial\Omega
\]
is satisfied. Then for any bounded weak solution $u$ of
\eqref{divTF2} in $(0,T)\times \Omega$ such that $u=g$ a.e. on $(0,T)\times
\partial\Omega$, there holds
\begin{equation} \label{Regparabound}
[u]_{C^{\frac{\alpha\epsilon}{2},\epsilon}([0,T]\times
\overline{\Omega}\,)}\le
C\Big(|u|_{L_\infty((0,T)\times \Omega)}+|u_0|_{B_{pp}^{2-\frac{2}{p\alpha}}({\Omega})}+|f|_{L_r((0,T);L_q(\Omega))}+|g|_{Y}\Big),
\end{equation}
where
$\epsilon=\epsilon(|A|_\infty,\nu,\alpha,p,r,q,d,\Omega)$ and
$C=C(|A|_\infty,\nu,\alpha,p,r,q,d,\Omega,T)$ are positive constants.
\end{satz}
The proof uses Theorem \ref{interiorHoelder} and extension techniques both in space and time, together with the maximal regularity result Theorem \ref{MRRand}. This explains the regularity required for the initial and boundary data.

The regularity condition (Hf) imposed on the right-hand side $f$ in Theorem \ref{interiorHoelder} and Theorem \ref{Regparaboundary} cannot be weakened significantly. In fact, given $f\in L_r((0,T);L_q(\Omega))$
the best possible regularity for the solution $u$ (in general) is that of maximal 
$L_r((0,T);L_q(\Omega))$-regularity, that is
\[
u\in H^\alpha_r((0,T);L_q(\Omega))\cap L_r((0,T);H^2_q(\Omega)). 
\]
By the mixed derivative theorem (cf. \cite{Sob}), we have
\[
u\in H^{\alpha(1-\zeta)}_r((0,T);H^{2\zeta}_q(\Omega))\quad \mbox{for all}\;\zeta\in [0,1].
\]
Now observe that the condition
\[
\frac{1}{\alpha r}\,+\,\frac{d}{2q}\,<1
\]
from (Hf) just ensures the existence of a $\zeta\in (0,1)$ such that $\alpha(1-\zeta)-\frac{1}{r}>\delta$
and $2\zeta -\frac{d}{q}>\delta$ for some $\delta>0$, which implies H\"older continuity of $u$ by
Sobolev embedding.

Another important result in the De Giorgi-Nash-Moser theory for time fractional diffusion equations is the {\em weak Harnack inequality} due to Zacher \cite{ZwH}. To formulate the result, recall that $B(x,r)$ denotes the open ball with
radius $r>0$ centered at $x\in \iR^d$, and $\lambda_d$ stands for the
Lebesgue measure in $\iR^d$. For $\delta\in(0,1)$, $t_0\ge 0$,
$\tau>0$, and a ball $B(x_0,r)$, we define the boxes
\begin{align*}
Q_-(t_0,x_0,r)&=(t_0,t_0+\delta\tau r^{2/\alpha})\times B(x_0,\delta r),\\
Q_+(t_0,x_0,r)&=(t_0+(2-\delta)\tau r^{2/\alpha},t_0+2\tau
r^{2/\alpha})\times B(x_0,\delta r).
\end{align*}
We have now the following result for the equation
\begin{equation} \label{divTF3}
\partial_t^\alpha (u-u_0)-\mbox{div}\,\big(A(t,x)\nabla u\big)=0,\quad t\in (0,T),\,x\in
\Omega.
\end{equation}
\begin{satz} \label{localweakHarnack}
Let $\alpha\in(0,1)$, $T>0$, and $\Omega\subset \iR^d$ be a bounded
domain. Let $u_0\in L_2(\Omega)$ and suppose that the assumption (HA) is satisfied. Let
further $\delta\in(0,1)$, $\eta>1$, and $\tau>0$ be fixed. Then for
any $t_0\ge 0$ and $r>0$ with $t_0+2\tau r^{2/\alpha}\le T$, any
ball $B(x_0,\eta r)\subset\Omega$, any
$0<p<\frac{2+d\alpha}{2+d\alpha-2\alpha}$, and any nonnegative weak
supersolution $u$ of \eqref{divTF3}  in $(0,t_0+2\tau
r^{2/\alpha})\times B(x_0,\eta r)$ with $u_0\ge 0$ in $B(x_0,\eta
r)$, there holds
\begin{equation} \label{localwHarnackF}
\Big(\frac{1}{\lambda_{d+1}\big(Q_-(t_0,x_0,r)\big)}\,\int_{Q_-(t_0,x_0,r)}u^p\,d\lambda_{d+1}\Big)^{1/p}
\le C \einf_{Q_+(t_0,x_0,r)} u,
\end{equation}
where the constant $C=C(\nu,|A|_\infty,\delta,\tau,\eta,\alpha,d,p)$.
\end{satz}
Theorem \ref{localweakHarnack} says that nonnegative weak
supersolutions of \eqref{divTF3} with $u_0\ge 0$ satisfy a weak form of the Harnack
inequality in the sense that we do not have an estimate for the
supremum of $u$ on $Q_-(t_0,x_0,r)$ but only an $L_p$ estimate. 
It is also shown in \cite{ZwH} that
the critical exponent
$\frac{2+d\alpha}{2+d\alpha-2\alpha}$ is optimal, i.e. the
inequality in general fails to hold for $p\ge
\frac{2+d\alpha}{2+d\alpha-2\alpha}$.

Theorem \ref{localweakHarnack} can be regarded as the time fractional
analogue of the corresponding result in the classical parabolic case
$\alpha=1$, see e.g. \cite[Theorem 6.18]{Lm} and \cite{Trud}.
Sending $\alpha\to 1$, the critical exponent
tends to $1+2/d$, which coincides with the well-known critical exponent for the
heat equation. As pointed out in \cite{ZwH}, the statement of
Theorem \ref{localweakHarnack} remains valid for (appropriately
defined) weak supersolutions of \eqref{divTF3} with $u_0\ge 0$ on $(t_0,t_0+2\tau
r^{2/\alpha})\times B(x_0,\eta r)$ which are nonnegative on
$(0,t_0+2\tau r^{2/\alpha})\times B(x_0,\eta r)$. We also remark that the global
positivity assumption cannot be replaced by a local one, as simple
examples show, cf. \cite{Za3}. This significant difference to the
case $\alpha=1$ is due to the non-local nature of
$\partial_t^\alpha$. The same phenomenon is known for
integro-differential operators like $(-\Delta)^\alpha$ with
$\alpha\in (0,1)$, see e.g. \cite{Kass1}.

The proof of Theorem \ref{localweakHarnack} relies on suitable {\it a priori}
estimates for powers of $u$ and logarithmic estimates, which are derived by means
of the fundamental identity for the
regularized fractional derivative. It further uses Moser's iteration
technique and an elementary but subtle lemma of Bombieri and Giusti
\cite{BomGiu} (see also \cite[Lemma 2.2.6]{SalCoste})
which allows to avoid the rather technically involved
approach via $BMO$-functions.

From the weak Harnack inequality one can easily derive the strong maximum principle for weak subsolutions of \eqref{divTF3}, see \cite[Theorem 5.1]{ZwH}.
\begin{satz} \label{strongmax}
Let $\alpha\in(0,1)$, $T>0$, and $\Omega\subset \iR^d$ be a bounded
domain. Let $u_0\in L_2(\Omega)$ and suppose that the assumption (HA) is satisfied. Let $u\in
W_\alpha$ be a weak subsolution of \eqref{divTF3} in $(0,T)\times \Omega$ and
assume that $0\le \esup_{(0,T)\times \Omega}u<\infty$ and that $\esup_{\Omega}
u_0\le \esup_{(0,T)\times \Omega}u$. Then, if for some cylinder
$Q=(t_0,t_0+\tau r^{2/\alpha})\times B(x_0,r)\subset (0,T)\times \Omega$ with
$t_0,\tau,r>0$ and $\overline{B(x_0,r)}\subset \Omega$ we have
\begin{equation} \label{strrel}
\esup_{Q}u \,=\,\esup_{(0,T)\times \Omega}u,
\end{equation}
the function $u$ is constant on $(0,t_0)\times \Omega$.
\end{satz}

It is an interesting problem, whether nonnegative weak solutions of \eqref{divTF3} with $u_0\ge 0$, satisfy the (full) Harnack inequality. The latter
 means that \eqref{localwHarnackF}
  holds with $p=\infty$, that is the term on the left is replaced by $\esup_{Q_-} u$.
Very recently, the author and coauthors \cite{DKSZ} observed that in contrast to the classical case $\alpha=1$, the
full Harnack inequality (in the form described before) fails to hold in general in the time fractional case if the space dimension $d$
is at least 2, even in the case where the elliptic operator is the Laplacian. A corresponding counterexample can be found
in \cite{DKSZ}. Its construction uses the fact that for $d\ge 2$ the fundamental solution $Z(t,x)$ of the equation has a singularity at
$x=0$ for all $t>0$. The one-dimensional case is still an open problem. It is conjectured that the Harnack inequality is true in this case. The author could show that the Harnack inequality holds in the purely time dependent case '$d=0$', that is
in the case without elliptic operator, see \cite{Za3}. 

We conclude this section by illustrating the strength of the described regularity results. 
Theorem \ref{Regparaboundary} provides the key estimate to prove the global strong
solvability of the following {\it quasilinear} time fractional diffusion problem
\begin{equation} \label{quasiprob}
\left\{
\begin{array}{r@{\;=\;}l@{\;}l}
\partial_t^\alpha (u-u_0)-\mbox{div}\,\big(A(u)\,\nabla u\big) &  f,\quad &
 t\in (0,T),\,x\in \Omega\\
u|_{\partial\Omega} &  g, \quad &  t\in (0,T),\,x\in \partial\Omega\\
u|_{t=0} &  u_0,\quad &  x\in \Omega.
\end{array}
\right.
\end{equation}
Letting $p>d+\frac{2}{\alpha}$ we assume that
\begin{itemize}
\item [{\bf (Q1)}] $\;g\in
B_{pp}^{\alpha(1-\frac{1}{2p})}((0,T);L_p(\partial\Omega))\cap
L_p((0,T);B_{pp}^{2-\frac{1}{p}}(\partial\Omega))$,\\ $\;f\in L_p((0,T);L_p(\Omega))$, $u_0\in
B_{pp}^{2-\frac{2}{p\alpha}}(\Omega)$, and $u_0|_{\partial\Omega}=g|_{t=0}$ on
$\partial\Omega$;
\item [{\bf (Q2)}] $\;A\in C^1(\iR;\mbox{Sym}\{d\})$, and
there exists $\nu>0$ such that $(A(y)\xi|\xi)\ge \nu|\xi|^2$
for all $y\in \iR$ and $\xi\in \iR^d$.
\end{itemize}
Here Sym$\{d\}$ denotes the space of $d$-dimensional real symmetric
matrices. 

The following result has been established in \cite{ZaG}.
\begin{satz} \label{quasilinear}
Let $\Omega\subset \iR^d$ ($d\ge 2$) be a bounded domain with
$C^2$-smooth boundary. Let $\alpha\in (0,1)$, $T>0$ be an
arbitrary number, $p>d+\frac{2}{\alpha}$, and suppose that the
assumptions (Q1) and (Q2) are satisfied. Then the problem
(\ref{quasiprob}) possesses a unique strong solution $u$ in the
class
\[
u\in H^\alpha_p((0,T);L_p(\Omega))\cap L_p((0,T);H^2_p(\Omega)).
\]
\end{satz}
\section{Decay estimates for bounded domains} \label{bdddomain}
Let $\alpha\in (0,1)$, $\Omega\subset \iR^d$ be a bounded domain, $u_0\in L_2(\Omega)$ and consider
the problem
\begin{equation} \label{weakProb0}
\left\{
\begin{array}{r@{\;=\;}l@{\;}l}
\partial_t^\alpha (u-u_0)-\mbox{div}\,(A\,\nabla u) &  0,\quad &
 t>0,\,x\in \Omega\\
u|_{\partial\Omega} &  0, \quad &  t>0,\,x\in \partial\Omega\\
u|_{t=0} &  u_0,\quad &  x\in \Omega,
\end{array}
\right.
\end{equation}
where the coefficient matrix $A=A(t,x)$ is assumed to satisfy the parabolicity condition (HA). From Theorem
\ref{weaksol} we know that \eqref{weakProb0} has a unique weak solution 
$u$ on $(0,T)\times \Omega$ for each $T>0$. In this sense, $u$ is a global (in time) weak solution of \eqref{weakProb0}. 
We are now interested in the long-time behaviour of $u$, in particular in decay estimates for the $L_2(\Omega)$-norm of
$u$.  

Let us first consider the special case $A=I$, i.e.\ the case of the Laplacian.
Let $\{\phi_n\}_{n=1}^\infty
\subset \oH^1_2(\Omega)$ be an orthonormal basis of $L_2(\Omega)$ consisting of eigenfunctions of the negative
Dirichlet Laplacian with eigenvalues $\lambda_n>0$,
$n\in \iN$, and denote by $\lambda_1$ the smallest such eigenvalue.
Further, we define for $\mu\ge 0$ the so-called {\em relaxation function} $s_\mu: [0,\infty) \rightarrow \iR$ as the
solution of the Volterra equation
\begin{equation} \label{smu}
s_\mu(t)+\mu (g_\alpha\ast s_\mu)(t)=1,\quad t\ge 0.
\end{equation}
Note that $s_0\equiv 1$ and that \eqref{smu} is equivalent to the integro-differential equation
\[
\partial_t^\alpha(s_\mu-1)(t)+\mu s_\mu(t)=0,\quad t>0,\quad s_\mu(0)=1.
\]
It is known that for all $\mu\ge 0$ the function $s_\mu$ is positive and nonincreasing, $s_\mu\in H^1_{1,loc}(\iR_+)$,
and $\partial_\mu s_\mu(t)\le 0$; this follows e.g.\ from the theory of completely positive kernels, described in
\cite{JanI}, see also \cite{GLS}. 
Alternatively, one can argue with the well-known formula
\[
s_\mu(t)=E_\alpha(-\mu t^\alpha),\quad \mbox{where}\;\; E_\alpha(z):=\sum_{j=0}^\infty \frac{z^j}
{\Gamma(\alpha j+1)},\,z\in \iC,
\]
is the  Mittag-Leffler function, see e.g.\ \cite{KST}.

Then the solution $u$ of \eqref{weakProb0} with $A=I$ can
be represented via Fourier series as
\begin{equation} \label{uformel}
u(t,x)=\sum_{n=1}^\infty s_{\lambda_n}(t)\,(u_0|\phi_n)\phi_n(x),\quad t\ge 0,\,x\in \Omega,
\end{equation}
where $(\cdot|\cdot)$ stands for the standard inner product in $L_2(\Omega)$, cf.\ \cite[Section 1]{VZ} and \cite[Theorem 4.1]{NSY}. By Parseval's identity and since $\partial_\mu s_\mu\le 0$, it follows from (\ref{uformel}) that
\begin{align*}
|u(t,\cdot)|_{L_2(\Omega)}^2 & =\sum_{n=1}^\infty s_{\lambda_n}^2(t)\,|(u_0|\phi_n)|^2
\le s_{\lambda_1}^2(t) \sum_{n=1}^\infty|(u_0|\phi_n)|^2
= s_{\lambda_1}^2(t) |u_0|_{L_2(\Omega)}^2,
\end{align*}
and thus
\begin{equation} \label{udecay}
|u(t,\cdot)|_{L_2(\Omega)}\le s_{\lambda_1}(t)|u_0|_{L_2(\Omega)},\quad t\ge 0,
\end{equation}
cf.\ \cite{VZ}. This decay estimate is optimal as the example $u_0=\phi_1$ with solution $u(t,x)
=s_{\lambda_1}(t)\phi_1(x)$ shows. It is further known (see e.g.\ \cite[Remark 6.1]{VZ}) that
\[
\frac{1}{1+\mu \Gamma(1-\alpha)t^\alpha }\,\le s_\mu(t)\le \,\frac{1}{1+\mu \Gamma(1+\alpha)^{-1}t^\alpha}\,,\quad
t>0.
\]
This shows that, in contrast to the case $\alpha=1$, where $s_\mu(t)=e^{-\mu t}$, we only have an {\it algebraic decay}
with rate $t^{-\alpha}$ (up to some bounded positive factor) as $t\to \infty$. 

In the general case with {\em rough coefficients} we have the following result due to Vergara and Zacher \cite[Corollary 1.1]{VZ}.
\begin{satz} \label{decayA}
Let $\alpha\in (0,1)$, $\Omega\subset \iR^d$ be a bounded domain, $u_0\in L_2(\Omega)$ and assume that
(HA) is fulfilled. Then the global weak solution $u$ of \eqref{weakProb0} satisfies the estimate
\begin{equation} \label{korresest}
|u(t,\cdot)|_{L_2(\Omega)}\le s_{\nu \lambda_1}(t)\,|u_0|_{L_2(\Omega)},\quad
\mbox{a.a.}\,t>0.
\end{equation}
\end{satz} 
Theorem \ref{decayA} shows that the $L_2(\Omega)$-norm of the solution $u(t,\cdot)$ 
decays at least as fast as the relaxation function $s_\mu(t)$ with $\mu=\nu \lambda_1$.
This decay estimate is again optimal as the special case $A=\nu I$ shows, in fact specializing further to $\nu=1$ we recover the estimate (\ref{udecay}).

The proof of Theorem \ref{decayA} is based on energy estimates and the $L_p$-norm inequality, see Theorem \ref{Lpnorm}.
The basic idea is as follows.
Testing (formally) the PDE with $u$, integrating over $\Omega$, and using $A\ge \nu I$ as well as Poincar\'e's inequality we obtain
\[
\int_\Omega u \,\partial_t^\alpha(u-u_0)\,dx +\nu \lambda_1 \int_\Omega |u|^2\,dx\le 0,\quad t>0.
\]
By \eqref{LPU} this implies (with $|u(t)|_{L_2(\Omega)}:=|u(t,\cdot)|_{L_2(\Omega)}$)
\[
|u(t)|_{L_2(\Omega)}\partial_t^\alpha\Big( |u(\cdot)|_{L_2(\Omega)}-|u_0|_{L_2(\Omega)}\Big)(t)+\nu\lambda_1|u(t)|_{L_2(\Omega)}^2\le 0,\quad t>0.
\]
Assuming $|u(t)|_{L_2(\Omega)}>0$ we thus arrive at the fractional differential inequality
\[
\partial_t^\alpha\big(|u|_{L_2(\Omega)}-|u_0|_{L_2(\Omega)}\big)(t)+\nu\lambda_1|u(t)|_{L_2(\Omega)}\le 0,\quad t>0,
\]
which implies (\ref{korresest}), by a comparison principle argument. The rigorous proof in the weak setting requires 
much more effort, in particular, the problem has to be regularized suitably in time.

We point out that Theorem \ref{decayA} can be generalized to a much wider class of subdiffusion equations, which
covers e.g.\ equations of distributed order, see \cite{VZ}.
\section{Decay estimates in the full space case}
In this section we consider the classical time fractional diffusion equation in $\iR^d$,
\begin{equation} \label{fracdiff}
\left\{
\begin{array}{r@{\;=\;}l@{\;}l}
\partial_t^\alpha (u-u_0)-\Delta u &  0,\quad &
 t>0,\,x\in \iR^d\\
u|_{t=0} &  u_0,\quad &  x\in \iR^d,
\end{array}
\right.
\end{equation}
where again $\alpha\in (0,1)$.
Under appropriate conditions on the initial value $u_0$, the solution of (\ref{fracdiff}) can be represented as
\begin{equation} \label{solformulaTF}
u(t,x)=\int_{\iR^d} Z(t,x-y)u_0(y)\,dy,
\end{equation}
where $Z$ denotes the fundamental solution corresponding to (\ref{fracdiff}),
see \cite{Koch}. It is known (see e.g.\ \cite{Koch90}, \cite{SchnWyss}) that
\[
Z(t,x)=\pi^{-\frac{d}{2}}t^{\alpha-1}|x|^{-d} H^{20}_{12}\big(\frac{1}{4}|x|^2t^{-\alpha}
\big|^{(\alpha,\alpha)}_{(d/2,1), (1,1)}\big),\quad t>0,\,x\in \iR^d\setminus\{0\},
\]
where $H$ denotes the Fox $H$-function (\cite{KST,KiSa}). $Z(t,x)$ is nonnegative and $|Z(t,\cdot)|_{L_1(\iR^d)}=1$
for all $t>0$, see e.g.\ \cite[Section 2]{KSVZ}.

In what follows we write $f \star g$ for the convolution in $\iR^d$ of the functions $f,g$. Given $u_0\in L_2(\iR^d)$ we do not have in general any decay for $|Z(t,\cdot)\star u_0|_{L_2(\iR^d)}$, like in the case of the heat equation ($\alpha=1$). Now suppose that $u_0\in L_2(\iR^d)\cap L_1(\iR^d)$. Then it is well-known that $u(t,\cdot):=Z_H(t,\cdot)\star u_0$, where $Z_H$
denotes the classical heat kernel, decays in the $L_2$-norm as
\[
|u(t,\cdot)|_2 \lesssim t^{-\frac{d}{4}},\quad t>0,
\]
and this estimate is the best one can obtain in general (see e.g.\ \cite{BS}). Here $|v|_2:=|v|_{L_2(\iR^d)}$, and $v(t)\lesssim w(t),\,t>0$ means that there exists a constant $C>0$ such that $v(t)\le Cw(t)$, $t>0$. In the case of time fractional diffusion 
we have the following surprising result, cf.\ \cite[Corollary 3.2, Theorem 4.1]{KSVZ}.
\begin{satz} \label{decayfull2}
Let $d\in \iN$ and $u_0\in L_1(\iR^d)\cap L_2(\iR^d)$ and $u(t)=Z(t)\star u_0$. Then
\begin{align}
|u(t)|_2\,& \lesssim\,t^{-\min\{\frac{\alpha d}{4},\alpha\}},\quad t>0,\;d\in \iN\setminus\{4\}, \label{L2decay}\\
|u(t)|_{2,\infty} & \lesssim\,t^{-\alpha},\quad t>0,\;d=4.\nonumber
\end{align}
Moreover, the estimate (\ref{L2decay}) is the best one can get in general.
\end{satz}
Whereas in the case $\alpha=1$ the decay rate increases with the dimension $d$, time fractional diffusion leads to the
phenomenon of a {\em critical dimension}, which is $d=4$ in this case. Below the critical dimension the rate increases with
$d$, the exponent being $\alpha$ times the one from the heat equation, while above the critical dimension the decay rate
is the same for all $d$, namely $t^{-\alpha}$. The reason why the decay rate does not increase any further with $d$ lies in the fact that $t^{-\alpha}$ (up to a constant) coincides with the decay rate in the case of a bounded domain and homogeneous Dirichlet boundary condition, see Section \ref{bdddomain}. This also shows that for $\alpha\in (0,1)$ the diffusion is so slow that in higher dimensions ($d$ above the critical dimension) restriction to a bounded domain and the requirement of a homogeneous Dirichlet boundary condition do not improve the rate of decay. This is markedly different in the classical diffusion case, where we always have exponential (and thus a better) decay in the case of a bounded domain.

The decay rates in Theorem \ref{decayfull2} can be proved in different ways, cf.\ \cite{KSVZ}. Using the analytic and asymptotic properties of $H$ (see e.g.\ \cite{Koch,Koch90}), which is a rather complicated object, one can derive sharp $L_p(\iR^d)$-estimates for $Z(t,\cdot)$ for $t>0$ and all
$1\le p<\kappa(d)$, where $\kappa(d):=d/(d-2)$, $d\ge 3$, and $\kappa(1)=\kappa(2)=\infty$. For $d\ge 3$ one also
finds that $|Z(t,\cdot)|_{\kappa(d),\infty} \lesssim\,t^{-\alpha}$. These estimates and Young's inequality for convolutions
then yield the desired decay rates. Alternatively, one can employ tools from Harmonic Analysis such as Plancherel's theorem
and argue with properties of the Fourier transform of $Z(t,\cdot)$, which coincides with $s_{|\xi|^2}(t)$ (up to a constant,
depending on the used definition of the Fourier transform), cf. Section \ref{bdddomain} for the definition of the relaxation
function $s_\mu(t)$.

Theorem \ref{decayfull2} is only a special case of more general results obtained in \cite{KSVZ}, which also provide
decay rates for the $L_p$-norm and allow for a wider class of subdiffusion equations. 

The subsequent result states that for integrable initial data $u_0$ the asymptotic behaviour of  $Z(t)\star u_0$
is described by a multiple of $Z(t,x)$, see \cite[Theorem 3.6]{KSVZ}. We set $\kappa_1(d):=d/(d-1)$, $d\ge 2$, and
$\kappa_1(1):=\infty$.
\begin{satz} Let $d\in \iN$ and $1\le p<\kappa_1(d)$. Let further $u_0\in L_1(\iR^d)$ and set $M=\int_{\iR^d} u_0(y)\,dy$.

\noindent (i) There holds
\[
t^{\frac{\alpha d }{2}\,\left(1-\frac{1}{p}\right)}|u(t)-MZ(t)|_{p}\rightarrow 0,\quad\mbox{as}\;t\rightarrow \infty.
\]
(ii) Assume in addition that $||x|u_0|_1<\infty$. Then
\[
t^{\frac{\alpha d }{2}\,\left(1-\frac{1}{p}\right)}|u(t)-MZ(t)|_{p}\lesssim t^{ -\frac{\alpha}{2}},\quad t>0.
\]
Moreover, in the limit case $p=\kappa_1(d)$ we have
\[
t^{\frac{\alpha }{2}}|u(t)-MZ(t)|_{\kappa_1(d),\,\infty}\lesssim t^{ -\frac{\alpha}{2}},\quad t>0.
\]
\end{satz}

$\mbox{}$

\noindent {\bf Rico Zacher}, Ulm University, Institute of Applied Analysis, 89069 Ulm, Germany,
e-mail: rico.zacher@uni-ulm.de

\end{document}